\documentclass[12pt]{article}
\usepackage{amssymb}
\usepackage{amsmath,amsthm}
\usepackage[latin1]{inputenc}
\usepackage{hyperref}
\usepackage{color}
\usepackage{graphicx}
\hypersetup{colorlinks=true, linkcolor=blue, citecolor=blue,urlcolor=blue}

 \newtheorem{remark}{Remark}

 \newtheorem{lemma}[remark]{Lemma}
 \newtheorem{theorem}[remark]{Theorem}
 
 \newtheorem{corollary}[remark]{Corollary}

\title{Domination related parameters in rooted product graphs}

\author{Dorota Kuziak$^{\dag}$, Magdalena Lema\'nska$^{\ddag}$ and Ismael G. Yero$^{\S}$\\
\\
$^{\dag}${\small Departament d'Enginyeria Inform\`atica i Matem\`atiques,}\\
{\small Universitat Rovira i Virgili,}  {\small Av. Pa\"{\i}sos
Catalans 26, 43007 Tarragona, Spain.} \\
{\small dorota.kuziak\@@urv.cat}\\
$^{\ddag}$ {\small Department of Technical Physics and Applied Mathematics} \\ {\small Gda\'nsk University of Technology, ul. Narutowicza 11/12, 80-233 Gda\'nsk, Poland }\\{\small magda\@@mifgate.mif.pg.gda.pl}\\
$^{\S}${\small Departamento de Matem\'aticas, Escuela Polit\'ecnica Superior de Algeciras}\\
{\small Universidad de C\'adiz,} {\small
Av. Ram\'on Puyol, s/n, 11202 Algeciras, Spain.} \\ {\small
ismael.gonzalez\@@uca.es}
}

\date{}

\begin{document}

\maketitle

\begin{abstract}
A set $S$ of vertices of a graph $G$ is a dominating set in $G$ if
every vertex outside of $S$ is adjacent to at least one vertex
belonging to $S$. A domination parameter of $G$ is related to those
sets of vertices of a graph satisfying some domination property together
with other conditions on the vertices of $G$. Here, we investigate
several domination related parameters in rooted product graphs.
\end{abstract}

{\it Keywords:} Domination; Roman domination; domination related parameters; rooted product graphs.

{\it AMS Subject Classification Numbers:}  05C12; 05C76.

\section{Introduction}

Domination in graph constitutes a very important area in graph theory \cite{bookdom1}. An enormous quantity of researches about domination in graphs have been developed in the last years. Nevertheless, there are still several open problems and incoming researches about that. One interesting question in this area is related to the study of domination related parameters in product graphs. For instance, the Vizing's conjecture \cite{vizing1,vizing}, is one of the most popular open problems about domination in product graphs. The Vizing's conjecture states that the domination number of Cartesian product graphs is greater than or equal to the product of the domination numbers of the factor graphs. Moreover, several kind of domination related parameters have been studied in the last years. Some of the most remarkable examples are the following ones. The domination number of direct product graphs was studied in \cite{dom-direct,dom-direct-1,dom-direct-2}. The total domination number of direct product graphs was studied in \cite{tot-dom-direct}. The upper domination number of Cartesian product graphs was studied in \cite{upper-dom-cart}. The independence domination number of Kronecker product graphs was studied in \cite{ind-dom-kron}. Several domination related parameters of corona product graphs and the conjunction of two graphs were studied in \cite{dom-par-corona} and \cite{dom-par-conjunction}, respectively. The Roman domination number of lexicographic product graphs was studied in \cite{Rom-dom-lexicog}. According to the quantity of works devoted to the study of domination related parameters in product graphs it is noted that not only Vizing's conjecture is an interesting topic related to domination in product graphs. In this sense, in this paper we pretend to contribute with the study of some domination related parameters for the case of rooted product graphs.

We begin by establishing the principal terminology and notation
which we will use throughout the article. Hereafter $G=(V,E)$
represents an undirected finite graph without loops and multiple
edges with set of vertices $V$ and set of edges $E$. The order of
$G$ is $|V|=n(G)$ and the size $|E|=m(G)$ (If there is no ambiguity
we will use only $n$ and $m$). We denote two adjacent vertices
$u,v\in V$ by $u\sim v$ and in this case we say that $uv$ is an edge
of $G$ or $uv\in E$. For a nonempty set $X\subseteq V$ and a vertex
$v\in V$, $N_X(v)$ denotes the set of neighbors that $v$ has in $X$:
$N_X(v):=\{u\in X: u\sim v\}$ and the degree of $v$ in $X$ is
denoted by $\delta_{X}(v)=|N_{X}(v)|.$ In the case $X=V$ we will use
only $N(v)$, which is also called the open neighborhood of a vertex
$v\in V$, and $\delta(v)$ to denote the degree of $v$ in $G$. The
close neighborhood of a vertex $v\in V$ is $N[v]=N(v)\cup \{v\}$.
The minimum and maximum degrees of $G$ are denoted by $\delta$ and
$\Delta$, respectively. The subgraph induced by $S\subset V$ is
denoted by $\langle S\rangle $ and the complement of the set $S$ in
$V$ is denoted by $\overline{S}$. The distance
between two vertices $u,v\in V$ of $G$ is denoted  by $d_G(u,v)$ (or
$d(u,v)$ if there is no ambiguity).

The set of vertices $D\subset V$ is a {\em dominating set} of $G$ if for
every vertex $v\in \overline{D}$ it is satisfied that
$N_D(v)\ne \emptyset$. The minimum cardinality of any dominating set of $G$ is the {\em domination number} of $G$ and it is
denoted by $\gamma(G)$.  A set $D$ is a $\gamma(G)$-set if it is a dominating set and $|D|=\gamma(G)$. Throughout the article we follow the terminology and notation of \cite{bookdom1}.

Given a graph $G$ of order $n$ and a graph $H$ with root vertex
$v$, the rooted product $G\circ H$  is defined  as the graph obtained
from $G$ and $H$ by taking one copy of $G$ and $n$ copies of $H$
and identifying the vertex $u_i$ of $G$ with the vertex $v$ in the
$i^{th}$ copy of $H$ for every $1\leq i\leq n$ \cite{rooted-first}. If $H$ or $G$ is the singleton graph, then $G\circ H$ is equal to $G$ or $H$, respectively. In this sense, to obtain the rooted product $G\circ H$, hereafter we will only consider graphs $G$ and $H$ of orders greater than or equal to two. Figure \ref{p-4-c-3} shows the case of the rooted product graph $P_4\circ C_3$. Hereafter, we will denote by
$V=\{v_1,v_2,...,v_n\}$   the set of vertices of $G$ and by
$H_i=(V_i,E_i)$  the $i^{th}$ copy of $H$ in $G\circ H$.

\begin{figure}[h]
  \centering
  \includegraphics[width=0.3\textwidth]{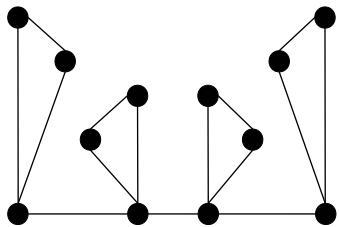}
  \caption{The rooted product graph $P_4\circ C_3$.}\label{p-4-c-3}
\end{figure}

It is clear that the value of every parameter of the
rooted product graph depends on the root of the graph $H$.
In the present article we present some results related to some domination parameters
in rooted product graphs.

\section{Domination number}

We begin with the following remark which will be useful into proving next results.

\begin{lemma}\label{rem1}
Let $G$ be a graph of order $n\ge 2$ and let $H$ be any graph with root $v$ and at least two vertices. If $v$ does not
belong to any $\gamma(H)$-set or $v$ belongs to every $\gamma(H)$-set, then  $$\gamma(G\circ H)=n\gamma(H).$$
\end{lemma}

\begin{proof}
If $A_i$ is a dominating set of minimum cardinality in $H_i=(V_i,E_i)$, $i\in \{1,...,n\}$, then it is clear that $\bigcup_{i=1}^nA_i$ is a dominating set in $G\circ H$. Thus $\gamma(G\circ H)\le n\gamma(H)$. Suppose $v$ does not
belong to any $\gamma(H)$-set. Let $S$ be a $\gamma(G\circ H)$-set and let $S_i=S\cap V_i$ for every $i\in \{1,...,n\}$. Notice that the set $S_i$ dominates all the vertices of $H_i$ except maybe the root $v_i$ which could be dominated by other vertex not in $H_i$.

If $v_j\notin S_j$ for some $j\in \{1,...,n\}$, then $S_j$ is a dominating set in $H_j-v_j$. So $\gamma(H_j-v_j)\le |S_j|$. Moreover, since $v_j$ does not belong to any $\gamma(H_j)$-set, it is satisfied that $\gamma(H_j-v_j)=\gamma(H_j)$. If  $|S_j|<\gamma(H)$, then we have that $\gamma(H_j-v_j)\le |S_j|<\gamma(H_j)=\gamma(H_j-v_j)$, a contradiction. On the other side, if $v_l\in S_l$ for some $l\in \{1,...,n\}$, then $S_l$ is a dominating set in $H_l$. So $\gamma(H_l)\le |S_l|$.  Therefore, $|S_i|\ge \gamma(H)$ for every $i\in \{1,...,n\}$ and we obtain that $\gamma(G\circ H)=n\gamma(H)$.

On the other hand, let us suppose $v$ belongs to every $\gamma(H)$-set. Thus, $v$ dominates at least one vertex in $H$ which is not dominated by any other vertex in every $\gamma(H)$-set and, as a consequence, $\gamma(H-v)\ge \gamma(H)$. As above, $S$ denotes a $\gamma(G\circ H)$-set and $S_i=S\cap V_i$ for every $i\in \{1,...,n\}$. If $v_j\notin S_j$ for some $j\in \{1,...,n\}$, then either $S_j$ is not a dominating set in $H_j$ (which is a contradiction) or $|S_j|\ge \gamma(H_j-v_j)\ge \gamma(H_j)$. On the contrary, if $v_j\in S_j$, then $S_j$ is a dominating set in $H_j$ and $|S_j|\ge \gamma(H_j)$. Therefore, we have that $|S|=\sum_{i=1}^n|S_i|\ge \sum_{i=1}^n\gamma(H_i)=n\gamma(H)$ and the proof is complete.
\end{proof}

\begin{theorem}
Let $G$ be a graph of order $n\ge 2$. Then for any graph $H$ with root $v$ and at least two vertices,
$$\gamma(G\circ H)\in \{n\gamma(H),n(\gamma(H)-1)+\gamma(G)\}.$$
\end{theorem}

\begin{proof}
It is clear that $\gamma(G\circ H)\le n\gamma(H)$ and also, from Lemma \ref{rem1}, there are rooted product graphs $G\circ H$ such
that $\gamma(G\circ H)=n\gamma(H)$. Now, let us  suppose that
$\gamma(G\circ H)<n\gamma(H)$. Let $V$ be the set of vertices of $G$
and let $V_i$, $i\in \{1,...,n\}$, be the set of vertices of the
copy $H_i$ of $H$ in $G\circ H$. Hence, if $S$ is a $\gamma(G\circ H)$-set, then there exists $j\in
\{1,...,n\}$ such that $|S\cap V_j|<\gamma(H)$. Notice that the set $S\cap V_j$ dominates all the vertices in $V_j$ excluding $v_j$. If $|S\cap V_j|<\gamma(H)-1$, then the set $(S\cap V_j)\cup \{v_j\}$ is a dominating set in $H_j$ and $|(S\cap V_j)\cup \{v_j\}|\le |(S\cap V_j)|+1<\gamma(H)$, which is a contradiction. So, $|S\cap V_i|\ge \gamma(H)-1$ for every $i\in \{1,...,n\}$.

Let $x$ be the number of copies $H_{j_1},H_{j_2},...,H_{j_x}$ of $H$ in which the vertex $v_{j_i}$ of $G$ is not dominated by $S\cap V_{j_i}$ ({\em i.e.}, $v_{j_i}$ is dominated by a vertex of $G$ belonging to other copy $H_l$, with $l\notin\{j_1,...,j_x\}$). On the contrary, let $y=n-x$ be the number of copies $H_{k_1},H_{k_2},...,H_{k_y}$ of $H$ in which the vertex $v_{k_i}$ of $G$ is dominated by $S\cap V_{k_i}$ or $v_{k_i}\in S$. Note that the $y$ vertices $v_{k_i}$ of $G$ satisfying the above property form a dominating set in $G$ and, as a consequence, $\gamma(G)\le y$. Since $n=x+y$, we have that $x\le n-\gamma(G)$. Also, notice that if the vertex $v_{k_i}$ of $G$ is dominated by $S\cap V_{k_i}$ or $v_{k_i}\in S$, then $S\cap V_{k_i}$ is a dominating set in $H_{k_i}$. So, $\gamma(H)\le |S\cap V_{k_i}|$ for every copy $H_{k_i}$ in which the vertex $v_{k_i}$ of $G$ is dominated by $S\cap V_{k_i}$ or $v_{k_i}\in S$. Thus we have the following.

\begin{align*}
\gamma(G\circ H)=|S|&=\left|\bigcup_{i=1}^{x}(S\cap V_{j_i})\cup\bigcup_{i=1}^{y}(S\cap V_{k_i})\right|\\
&=\sum_{i=1}^{x}|S\cap V_{j_i}|+\sum_{i=1}^{y}|S\cap V_{k_i}|\\
&\ge x(\gamma(H)-1)+y\gamma(H)\\
&=n\gamma(H)-x\\
&\ge n\gamma(H)-n+\gamma(G)\\
&=n(\gamma(H)-1)+\gamma(G).
\end{align*}

On the other side, let $A$ be a $\gamma(G\circ H)$-set. Since $\gamma(G\circ H)<n\gamma(H)$, there exists at least one copy $H_k$ of $H$ such that $|A\cap V_k|<\gamma(H)$, which implies $|A\cap V_k|\le \gamma(H)-1$. Since $A\cap V_k$ dominates all the vertices of $H_k$ except maybe the root $v_k$, we have that if $v_k\in A\cap V_k$, then $A\cap V_k$ is a dominating set in $H$, which is a contradiction. So, $v_k\notin A\cap V_k$. Now, as $|A\cap V_i|\ge \gamma(H)-1$ for every $i\in \{1,...,n\}$, we obtain that $|A\cap V_k|=\gamma(H)-1$. So, $A'=(A\cap V_k)\cup \{v_k\}$ is a $\gamma(H)$-set. Let us denote by $A'_i$, $i\in \{1,...,n\}$, the set of vertices of $A'-\{v_i\}$ in each copy $H_i$ of $G\circ H$.

Let $B$ be a $\gamma(G)$-set and let $D=\left(\bigcup_{i=1}^nA'_i\right)\cup B$. Since $A'_i$ dominates the vertices of $H_i-\{v_i\}$ for every $i\in \{1,...,n\}$ and $B$ dominates the vertices of $G$, we obtain that $D$ is a dominating set in $G\circ H$. Thus $$|D|=\sum_{i=1}^n|A'_i|+|B|=n(|A'|-1)+|B|=n(\gamma(H)-1)+\gamma(G).$$

Therefore, we obtain that $\gamma(G\circ H)\le n(\gamma(H)-1)+\gamma(G)$
and the result follows.
\end{proof}

\section{Roman domination number}

Roman domination number was defined by Stewart
in \cite{roman-1} and studied further by some researchers, for instance in \cite{roman}. Given a graph $G=(V,E)$, a map $f : V
\rightarrow \{0, 1, 2\}$ is a {\em Roman dominating function} for $G$ if for every vertex $v$ with $f(v) = 0$, there exists a
vertex $u\in N(v)$ such that $f(u) = 2$. The {\em weight} of a Roman
dominating function is given by $f(V) =\sum_{u\in V}f(u)$. The
minimum weight of a Roman dominating function on $G$ is called the
{\em Roman domination number} of $G$ and it is denoted by
$\gamma_R(G)$. A function $f$ is a $\gamma_R(G)$-function in a graph $G=(V,E)$ if it is a Roman dominating function and $f(V)=\gamma_R(G)$.

Let $f$ be a Roman dominating function on $G$ and let $B_0$,
$B_1$ and $B_2$ be the sets of vertices of $G$ induced
by $f$, where $B_i = \{v\in V\;:\; f(v) = i\}$. Frequently, a Roman dominating function $f$ is represented by the sets $B_0$,
$B_1$ and $B_2$, and it is common to denote  $f=(B_0,B_1, B_2)$. It is clear that for any
Roman dominating function $f$ on the graph $G=(V,E)$ of order $n$ we
have that $f(V)=\sum_{u\in V}f(u)=2|B_2|+|B_1|$  and
$|B_2|+|B_1|+|B_0|=n$. The following lemmas will be useful into proving other results in this section.

\begin{lemma}{\em \cite{roman}}\label{lema-roman}
For any graph $G$, $\gamma(G)\le \gamma_R(G)\le 2\gamma(G)$.
\end{lemma}

\begin{lemma}\label{lema-romano}
Let $G=(V,E)$ be a graph and let $f=(B_0,B_1,B_2)$ be a $\gamma_R(G)$-function. Then for every $v\in V$,
\begin{itemize}
\item[\rm{(i)}] If $v\in B_0$, then $\gamma_R(G)-1\le \gamma_R(G-v)\le \gamma_R(G)$.
\item[\rm{(ii)}] If $v\in B_1$, then $\gamma_R(G-v)=\gamma_R(G)-1$.
\item[\rm{(iii)}] If $v\in B_2$, then $\gamma_R(G)-1\le \gamma_R(G-v)\le\gamma_R(G)+\delta(v)-2$.
\end{itemize}
\end{lemma}

\begin{proof}
Let  $f'=(A_0,A_1,A_2)$ be a $\gamma_R(G-v)$-function. By making $f'(v)=1$ we have that $f'$ is a Roman dominating function in $G$. Thus \begin{equation}\label{lab-1}
\gamma_R(G)\le \gamma_R(G-v)+1.
\end{equation}
Now, if $v\in B_0$, then it is clear that $\gamma_R(G-v)\le \gamma_R(G)$ and (i) is proved.

Moreover, if $v\in B_1$, then $(B_0,B_1-\{v\},B_2)$ is a Roman dominating function in $G-v$. Thus $\gamma_R(G-v)\le \gamma_R(G)-1$. Therefore, by (\ref{lab-1}) we obtain (ii).

On the other hand, if $v\in B_2$, then $(B_0,B_1\cup (N(v)-B_2),B_2-\{v\})$ is a Roman dominating function in $G-v$. Thus
\begin{align*}
\gamma_R(G-v)&\le 2|B_2-\{v\}|+|B_1\cup (N(v)-B_2)|\\
&=2|B_2|-2+|B_1|+|N(v)-B_2|\\
&\le \gamma_R(G)+\delta(v)-2.
\end{align*}
Therefore, (iii) is proved.
\end{proof}

\begin{lemma}\label{lema-romano-1}
Let $G=(V,E)$ be a graph. If for every $\gamma_R(G)$-function $f=(B_0,B_1,B_2)$ is satisfied that $v\in B_0$, then $$\gamma_R(G-v)=\gamma_R(G).$$
\end{lemma}

\begin{proof}
From Lemma \ref{lema-romano} (i) we have that $\gamma_R(G-v)\le \gamma_R(G)$. If $\gamma_R(G-v)<\gamma_R(G)$, then there exists a $\gamma_R(G-v)$-function $h=(A_0,A_1,A_2)$ such that $h(V-\{v\})=\gamma_R(G-v)<\gamma_R(G)$, which leads to $h(V-\{v\})\le \gamma_R(G)-1$. If $h'$ is a function in $G$ such that for every $u\in V$, $u\ne v$, we have that $h'(u)=h(u)$ and $h'(v)=1$, then $h'$ is a Roman dominating function in $G$. Thus, $\gamma_R(G)\le h'(V)=h(V-\{v\})+1\le \gamma_R(G)$. So, $\gamma_R(G)=h'(V)=h(V-\{v\})+1=\gamma_R(G)$ and we have that $h'$ is a $\gamma_R(G)$-function such that $h'(v)=1$, which is a contradiction. Therefore, $\gamma_R(G-v)=\gamma_R(G)$.
\end{proof}

The Roman domination number of rooted product graphs is studied at
next.

\begin{theorem}\label{teo-roman}
Let $G$ be a graph of order $n\ge 2$. Then for any graph $H$ with root $v$ and at least two vertices,
$$n(\gamma_R(H)-1)+\gamma(G)\le \gamma_R(G\circ H)\le n\gamma_R(H).$$
\end{theorem}

\begin{proof}
It is clear that $\gamma_R(G\circ H)\le n\gamma_R(H)$. Let $V_i$ be the set of vertices of $H_i$ for every $i\in \{1,...,n\}$ and let $f=(B_0,B_1,B_2)$ be a $\gamma_R(G\circ H)$-function. Now, for every $i\in \{1,...,n\}$ and every $k\in \{0,1,2\}$, let $B^{(i)}_k=B_k\cap V_i$. Let $j\in \{1,...,n\}$. We consider the following cases.

Case 1: $v_j\in B^{(j)}_0$. If $N_{H_j}(v_j)\cap B^{(j)}_2\ne \emptyset$, then $f_j=(B^{(j)}_0-\{v_j\},B^{(j)}_1,B^{(j)}_2)$ is a Roman dominating function in $H_j-v_j$. On the contrary, if $N_{H_j}(v_j)\cap B^{(j)}_2=\emptyset$, then $v_j$ is adjacent to some vertex $v_k\in B^{(k)}_2$, with $k\ne j$ and, again $f_j=(B^{(j)}_0-\{v_j\},B^{(j)}_1,B^{(j)}_2)$ is a Roman dominating function in $H_j-v_j$. So, $\gamma_R(H_j-v_j)\le 2|B^{(j)}_2|+|B^{(j)}_1|$. By Lemma \ref{lema-romano} (i) we have that $\gamma_R(H_j-v_j)\ge \gamma_R(H)-1$. Thus $2|B^{(j)}_2|+|B^{(j)}_1|\ge \gamma_R(H)-1$.

Case 2: $v_j\in B^{(j)}_1$. Hence, it is clear that $f_j=(B^{(j)}_0,B^{(j)}_1-\{v_j\},B^{(j)}_2)$ is a Roman dominating function in $H_j-v_j$. So, $\gamma_R(H_j-v_j)\le 2|B^{(j)}_2|+|B^{(j)}_1|-1$. By Lemma \ref{lema-romano} (ii) we have that $\gamma_R(H_j-v_j)=\gamma_R(H)-1$. Thus $2|B^{(j)}_2|+|B^{(j)}_1|\ge \gamma_R(H)$.

Case 3: $v_j\in B^{(j)}_2$. Thus $f_j=(B^{(j)}_0,B^{(j)}_1,B^{(j)}_2)$ is a Roman dominating function in $H_j$. So, $2|B^{(j)}_2|+|B^{(j)}_1|\ge \gamma_R(H)$.

Now, let $V$ be the set of vertices of $G$ and let $A\subseteq V\cap B_0$ be the set of vertices of $G$ such that for every vertex $v_l\in A$ is satisfied that $N_{H_l}(v_l)\cap B^{(l)}_2=\emptyset$. So, every vertex $v_l\in A$ is dominated by some vertex in $(V-A)\cap  B^{(k)}_2$, with $k\ne l$. As a consequence, $V-A$ is a dominating set and $\gamma(G)\le n-|A|$. Since $A\subseteq V\cap B_0$, it is satisfied that $|A|$ equals at most the numbers of copies $H_j$ of $H$ such that $2|B^{(j)}_2|+|B^{(j)}_1|\ge \gamma_R(H)-1$ (those copies satisfying Case 1). Thus we have the following,

\begin{align*}
\gamma_R(G\circ H)&=2|B_2|+|B_1|\\
&=\sum_{i=1}^n(2|B^{(i)}_2|+|B^{(i)}_1|)\\
&=\sum_{i=1}^{n-|A|}(2|B^{(i)}_2|+|B^{(i)}_1|)+\sum_{i=1}^{|A|}(2|B^{(i)}_2|+|B^{(i)}_1|)\\
&\ge(n-|A|)\gamma_R(H)+|A|(\gamma_R(H)-1)\\
&=n\gamma_R(H)-|A|\\
&\ge n(\gamma_R(H)-1)+\gamma(G).
\end{align*}
Therefore the lower bound is proved.
\end{proof}

As the following proposition shows, the above bounds are tight.

\begin{theorem}
Let $G$ be a graph of order $n\ge 2$ and let $H$ be a graph with root $v$ and at least two vertices. Then,
\begin{itemize}
\item[\rm{(i)}] If for every $\gamma_R(H)$-function $f=(B_0,B_1,B_2)$ is satisfied that $f(v)=0$, then $$\gamma_R(G\circ H)=n\gamma_R(H).$$
\item[\rm{(ii)}] If there exist two $\gamma_R(H)$-functions $h=(B_0,B_1,B_2)$ and $h'=(B'_0,B'_1,B'_2)$ such that $h(v)=1$ and $h'(v)=2$, then $$\gamma_R(G\circ H)=n(\gamma_R(H)-1)+\gamma(G).$$
\end{itemize}
\end{theorem}

\begin{proof}
Let $f'=(B'_0,B'_1,B'_2)$ be a $\gamma_R(G\circ H)$-function and let $V_i$ be the set of vertices of $H_i$, $i\in \{1,...,n\}$. Now, for every $i\in \{1,...,n\}$, let $f_i=(B^{(i)}_0=B'_0\cap V_i,B^{(i)}_1=B'_1\cap V_i,B^{(i)}_2=B'_2\cap V_i)$. From Theorem \ref{teo-roman} we have that $\gamma_R(G\circ H)\le n\gamma_R(H)$. If $\gamma_R(G\circ H)<n\gamma_R(H)$, then there exists $j\in \{1,...,n\}$ such that $f_j(V_j)=2|B^{(j)}_2|+|B^{(j)}_1|<\gamma_R(H)$. So $f_j=(B^{(j)}_0,B^{(j)}_1,B^{(j)}_2)$ is not a Roman dominating function in $H_j$. If $f'(v_j)=1$ or $f'(v_j)=2$, then every vertex in $B^{(j)}_0$ is adjacent to a vertex in $B^{(j)}_2$ and, as a consequence, $(B^{(j)}_0,B^{(j)}_1,B^{(j)}_2)$ is a Roman dominating function in $H_j$, which is a contradiction. So $f'(v_j)=0$ and $f_j=(B^{(j)}_0-\{v_j\},B^{(j)}_1,B^{(j)}_2)$ is a $\gamma_R(H_j-v_j)$-function. Since $f(v)=0$ for every $\gamma_R(H)$-function, by Lemma \ref{lema-romano-1} we have that $2|B^{(j)}_2|+|B^{(j)}_1|=\gamma_R(H-v)=\gamma_R(H)$ and this is a contradiction. Therefore, $\gamma_R(G\circ H)=n\gamma_R(H)$ and (i) is proved.

To prove (ii), for every $i\in \{1,...,n\}$ we consider two $\gamma_R(H_i)$-functions $h_i=(A^{(i)}_0,A^{(i)}_1,A^{(i)}_2)$ and $h'_i=(B^{(i)}_0,B^{(i)}_1,B^{(i)}_2)$ such that $h_i(v_i)=1$ and $h'_i(v)=2$, and let $S$ be a $\gamma(G)$-set. Now, we define a function $g$ in $G\circ H$ in the following way.
\begin{itemize}
\item For every vertex $x$ belonging to a copy $H_j$ of $H$ such that the root $v_j\in S$ we make $g(x)=h'(x)$ (notice that $g(v_j)=2$).
\item For every vertex $y$, except the corresponding root, belonging to a copy $H_l$ of $H$ such that the root $v_l\notin S$, we make $g(x)=h(x)$.
\item For every root of every copy $H_l$ satisfying the conditions of the above item we make $g(x)=0$ (note that these vertices are adjacent to a vertex $w$ of $G$ for which $g(w)=2$).
\end{itemize}
Since every vertex $u\in V_j$ not in $G$, with $g(u)=0$, is adjacent to a vertex $u'$ such that $g(u')=2$ and also, every vertex $v_l$ of $G$, with $g(v_l)=0$, is adjacent to a vertex $v_k\in S$ with $g(v_k)=2$, we obtain that $g$ is a Roman dominating function in $G\circ H$. Thus
\begin{align*}
\gamma_R(G\circ H)&\le \sum_{i=1}^{|S|}(2|B^{(i)}_2|+|B^{(i)}_1|)+\sum_{i=1}^{n-|S|}(2|A^{(i)}_2|+|A^{(i)}_1|-1)\\
&=|S|\gamma_R(H)+(n-|S|)(\gamma_R(H)-1)\\
&=n(\gamma_R(H)-1)+|S|\\
&=\gamma(G)+n(\gamma_R(H)-1).
\end{align*}
Therefore, (ii) follows by Theorem \ref{teo-roman}.
\end{proof}

On the other hand, we can see that there are rooted product graphs for which the bounds of Theorem \ref{teo-roman} are not achieved.

\begin{theorem}
Let $G$ be a graph of order $n\ge 2$ and let $H$ be a graph with root $v$ and at least two vertices. If for every $\gamma_R(H)$-function $f$ is satisfied that $f(v)=1$, then
$$\gamma_R(G\circ H)=n(\gamma_R(H)-1)+\gamma_R(G).$$
\end{theorem}

\begin{proof}
Let $f=(B_0,B_1,B_2)$ be a $\gamma_R(H)$-function and let $f'=(B'_0,B'_1,B'_2)$ be a $\gamma_R(G)$-function. Now, let us define a function $h$ in $G\circ H$ such that if $u\ne v$, then $h(u)=f(u).$ Otherwise, $h(u)=f'(u)$. Since $f(v)=1$ for every $\gamma_R(H)$-function, it is satisfied that every vertex $x$ of $G\circ H$ with $h(x)=0$ is adjacent to a vertex $y$ in $G\circ H$ with $h(y)=2$. Thus $h$ is a Roman dominating function in $G\circ H$ and we have that
\begin{align*}
\gamma_R(G\circ H)&\le (2|B'_2|+|B'_1|)+\sum_{i=1}^n (2|B_2|+|B_1|-1)\\
&=n(\gamma_R(H)-1)+\gamma_R(G).
\end{align*}
On the other hand, let $V_i$, $i\in \{1,...,n\}$, be the set of vertices of the copy $H_i$ of $H$ in $G\circ H$ and let $V$ be the set of vertices of $G$. Now, let $g=(A_0,A_1,A_2)$ be a $\gamma_R(G\circ H)$-function and for every $i\in \{1,...,n\}$ let $g_i=(A^{(i)}_0=A_0\cap V_i,A^{(i)}_1=A_1\cap V_i,A^{(i)}_2=A_2\cap V_i)$. Since the root $v_i$ of $H_i$ satisfies that $f(v_i)=1$ for every $\gamma_R(H_i)$-function $f$, we have the following cases.

Case 1: If there exists $l\in \{1,...,n\}$ such that $g(v_l)=2$, then $g_l$ is a Roman dominating function in $H_l$, but it is not a $\gamma_R(H)$-function. Thus $\gamma_R(H_l)<2|A^{(l)}_2|+|A^{(l)}_1|$, which leads to $$\gamma_R(H_l)\le 2|A^{(l)}_2|+|A^{(l)}_1|-1=2|A^{(l)}_2-\{v_l\}|+|A^{(l)}_1|+1.$$

Case 2: If there exists $j\in \{1,...,n\}$ such that $g(v_j)=1$, then $g_j$ is a Roman dominating function in $H_j$ and $g'_j=(A^{(j)}_0,A^{(j)}_1-\{v_j\},A^{(j)}_2)$ is a Roman dominating function in $H_j-v_j$. Thus, by Lemma \ref{lema-romano} (ii), it is satisfied that $$\gamma_R(H_j)=\gamma_R(H_j-v_j)+1\le 2|A^{(j)}_2|+|A^{(j)}_1-\{v_j\}|+1.$$

Case 3: If there exists $i\in \{1,...,n\}$ such that $g_i(v_i)=0$, then we have one of the following possibilities:

\begin{itemize}
\item $g_i$ is not a Roman dominating function in $H_i$. So, $v_i$ should be adjacent to a vertex $v_j$, $j\ne i$, of $G$ such that $g_j(v_j)=2$. Moreover, $g'_i=(A^{(i)}_0-\{v_i\},A^{(i)}_1,A^{(i)}_2)$ is a Roman dominating function in $H_i-v_i$ and by Lemma \ref{lema-romano} (ii) it is satisfied that $\gamma_R(H_i)=\gamma_R(H_i-v_i)+1\le 2|A^{(i)}_2|+|A^{(i)}_1|+1$.

\item $g_i$ is a Roman dominating function in $H_i$. Since $f(v_i)=1$ for every $\gamma_R(H_i)$-function $f$, we have that $g_i(V_i)>\gamma_R(H_i)$. Let $f_i$ be a $\gamma_R(H_i)$-function. Now, by taking a function $g'$ on $G\circ H$, such that if $v\in V_i$, then $g'(v)=f'(v)$ and if $v\notin V_i$, then $g'(v)=g(v)$, we obtain that $g'$ is a Roman dominating function for $G\circ H$ and the weight of $g'$ is given by
\begin{align*}
g'\left(\bigcup_{j=1}^{n}V_j\right)&=g\left(\bigcup_{j=1,j\ne i}^{n}V_j\right)+f_i(V_i)\\
&=g\left(\bigcup_{j=1,j\ne i}^{n}V_j\right)+\gamma_R(H_i)\\
&<g\left(\bigcup_{j=1,j\ne i}^{n}V_j\right)+g_i(V_i)\\
&=g\left(\bigcup_{j=1}^{n}V_j\right)\\
&=\gamma_R(G\circ H).
\end{align*}
    and this is a contradiction.
\end{itemize}
As a consequence, we obtain that if $g_i(v_i)=0$, then $g_i$ is not a Roman dominating function in $H_i$. So, every vertex $v_l$ of $G$ for which $g(v_l)=0$ is adjacent to a vertex $v_k$, $k\ne l$, of $G$ such that $g(v_k)=2$ and it is satisfied that the function $g'=(X_0=A_0\cap V,X_1=A_1\cap V,X_2=A_2\cap V)$ is a Roman dominating function in $G$ and $\gamma_R(G)\le 2|X_2|+|X_1|$. Thus we have the following,

\begin{align*}
\gamma_R(G\circ H)&=2|A_2|+|A_1|\\
&=\sum_{v_i\in X_0}(2|A^{(i)}_2|+|A^{(i)}_1|)+\sum_{v_i\in X_1}(2|A^{(i)}_2|+|A^{(i)}_1|)+\sum_{v_i\in X_2}(2|A^{(i)}_2|+|A^{(i)}_1|)\\
&=\sum_{v_i\in X_0}(2|A^{(i)}_2|+|A^{(i)}_1|)+\sum_{v_i\in X_1}(2|A^{(i)}_2|+|A^{(i)}_1-\{v_i\}|)+\\
&\hspace*{0.6cm}+\sum_{v_i\in X_2}(2|A^{(i)}_2-\{v_i\}|+|A^{(i)}_1|)+|X_1|+2|X_2|\\
&\ge \sum_{v_i\in X_0}(\gamma_R(H_i)-1)+\sum_{v_i\in X_1}(\gamma_R(H_i)-1)+\sum_{v_i\in X_2}(\gamma_R(H_i)-1)+2|X_2|+|X_1|\\
&\ge \sum_{i=1}^n(\gamma_R(H_i)-1)+\gamma_R(G)\\
&=n(\gamma_R(H)-1)+\gamma_R(G).
\end{align*}
Therefore the result follows.
\end{proof}

\section{Independent domination number}

A set of vertices $S$ of a graph $G$ is {\em independent} if the subgraph induced by $S$ has no edges. The maximum cardinality of an independent set in $G$ is called the {\em independence number} of $G$ and it is denoted by $\alpha(G)$. A set $S$ is a $\alpha(G)$-set if it is independent and $|S|=\alpha(G)$.
A set of vertices $D$ of a graph $G$ is an {\em independent dominating set} in $G$ if $D$ is a dominating set and the subgraph $\langle D\rangle$ induced by $D$ is independent in $G$ \cite{berge}. The minimum cardinality of any independent dominating set in $G$ is called the {\em independent domination
number} of $G$ and it is denoted by $i(G)$. A set $D$ is a $i(G)$-set if it is an independent dominating set and $|D|=i(G)$. At next we study the independent domination number of rooted product graphs and we begin by studying the independence number.

\begin{lemma}\label{lema-independence}
Let $v$ be any vertex of a graph $G$. If $v$ belongs to every $\alpha(G)$-set, then
$\alpha(G)\ge \alpha(G-v)+1$.
\end{lemma}

\begin{proof}
Let $S$ be a $\alpha(G-v)$-set. Since $S$ is still independent in $G$, we have $\alpha(G)\ge |S|$. If $\alpha(G)=|S|$, then $S$ is a $\alpha(G)$-set and $v\notin S$, a contradiction. So, $\alpha(G)\ge \alpha(G-v)+1$.
\end{proof}

\begin{theorem}
For any graph $G$ of order $n\ge 2$ and any graph $H$ with root $v$ and at least two vertices,
\begin{itemize}
\item[\rm{(i)}] If there is a $\alpha(H)$-set not containing the root $v$, then $\alpha(G\circ H)=n\alpha(H)$.
\item[\rm{(ii)}] If the root $v$ belongs to every $\alpha(H)$-set, then $\alpha(G\circ H)=n(\alpha(H)-1)+\alpha(G)$.
\end{itemize}
\end{theorem}

\begin{proof}
Let $S_i$, $i\in \{1,...,n\}$, be a $\alpha(H_i)$-set not containing the root $v_i$. Hence, $\bigcup_{i=1}^nS_i$ is independent in $G\circ H$. Thus $\alpha(G\circ H)\ge n\alpha(H)$. If $\alpha(G\circ H)>n\alpha(H)$, then there exists $j\in \{1,...,n\}$ such that $|S_j|>\alpha(H)$ and $S_j$ is independent, a contradiction. Therefore, $\alpha(G\circ H)=n\alpha(H)$.

On the other hand, suppose the root $v$ belongs to every $\alpha(H)$-set. Let $A_i$ be a $\alpha(H_i)$-set and let $B$ be a $\alpha(G)$-set. Since $v_i\in A_i$ for every $i\in \{1,...,n\}$, by taking $A=B\cup (\bigcup_{i=1}^n A_i-\{v_i\})$ we have that $A$ is independent in $G\circ H$. Thus
$$\alpha(G\circ H)\ge |A|=|B|+\sum_{i=1}^n|A_i-\{v_i\}|=n(\alpha(H)-1)+\alpha(G).$$
Now, let $V_i$, $i\in \{1,...,n\}$, be the set of vertices of the copy $H_i$ of $H$ in $G\circ H$ and let $V$ be the set of vertices of $G$. Let $X$ be a $\alpha(G\circ H)$-set and let $X_i=X\cap (V_i-\{v_i\})$ for every $i\in \{1,...,n\}$ and let $Y=V\cap X$. Notice that $Y$ and $X_i$ are independent sets. So, $\alpha(H_i-v_i)\ge |X_i|$ and $\alpha(G)\ge |Y|$ and by Lemma \ref{lema-independence} we have that $|X_i|\le \alpha(H_i)-1$. Thus
$$\alpha(G\circ H)=|Y|+\sum_{i=1}^n|X_i|\le \alpha(G)+\sum_{i=1}^n(\alpha(H_i)-1)=\alpha(G)+n(\alpha(H)-1).$$
Therefore, the proof is complete.
\end{proof}

\begin{lemma}\label{lema-indep-1}
Let $G=(V,E)$ be a graph. Then for every set of vertices $A\subset V$, $$i(G-A)\ge i(G)-|A|.$$
\end{lemma}

\begin{proof}
Let us suppose $i(G-A)<i(G)-|A|$. So, there exists an independent dominating set $S\subset V-A$ in $G-A$ such that $|S|<i(G)-|A|$. Let $v\in A$. If $N_S(v)\ne \emptyset$, then $v$ is independently dominated by the set $S$ in $G$. On the contrary, if $N_S(v)=\emptyset$, then the set $S\cup \{v\}$  is still independent. So, by adding those vertices which maintain the independence in the set $S$ we obtain a set $S'$ which is independent and dominating in $G$ and we have that $i(G)\le |S'|\le |S|+|A|<i(G)-|A|+|A|=i(G)$, which is a contradiction. Therefore, $i(G-A)\ge i(G)-|A|$.
\end{proof}

\begin{lemma}\label{lemma_G-v_id}
If $v$ does not belong to any $i(G)$-set, then
$$i(G-v) = i(G).$$
\end{lemma}

\begin{proof}
Let $S$ be an $i(G)$-set. Since $v\notin S$, $S$ is still independent and dominating in $G-v$. So, $i(G-v)\le i(G)$. On the other hand, let $A$ be an $i(G-v)$-set. Let us suppose that $|A| < i(G)$. So, $|A|\le i(G)-1$. If $N_{A}(v) = \emptyset$ in $G$, then $A\cup \{v\}$ is independent and dominating in $G$. So, $i(G)\le |A\cup \{v\}|=|A|+1\le i(G)$. Thus $|A+\{v\}|=i(G)$ and this is a contradiction because $v$ does not belong to any $i(G)$-set. On the contrary, if $N_{A}(v) \ne \emptyset$, then $A$ is independent and dominating in $G$, which is a contradiction ($|A|<i(G)$). So, $|A|\ge i(G)$. Therefore, $i(G-v) = |A| \ge i(G)$ and the result follows.
\end{proof}

\begin{theorem}\label{bounds-ind-dom}
Let $G=(V,E)$ be a graph of order $n\ge 2$ and let $H$ be a graph with root $v$ and at least two vertices. Then
$$n(i(H)-1)+i(G)\le i(G\circ H)\le i(H)\alpha(G) + i(H-v)(n - \alpha(G)).$$
\end{theorem}

\begin{proof}
Let $S$ be an $i(G\circ H)$-set and let $S_i=S\cap V_i$, $i\in \{1,...,n\}$. If $v\in S_j$ for some $j\in \{1,...,n\}$, then $S_j$ is an independent dominating set in $H_j$. So, $|S_j|\ge i(H)$. On the contrary, if $v\notin S_k$ for some $k\in \{1,...,n\}$, then $S_k$ independently dominates all vertices of $H_k-v$.  So, $S_k$ is an independent dominating set in $H_k-v$ and by Lemma \ref{lema-indep-1} we have that $|S_k|\ge i(H_k-v)\ge i(H)-1$. If $|S_j|=i(H_j)-1$ for some $j\in \{1,...,n\}$, then $v$ is not independently dominated by $S_j$. Also, if $v$ is independently dominated by $S_l$ for some $l\in \{1,...,n\}$, then $|S_l|\ge i(H_l)$. Let $A=S\cap V$ and let $B\subset V$ be the set of vertices of $G$ such that every vertex $u_i\in B$ is independently dominated by a vertex not in $G$. Notice that $A$ is an independent dominating set in $G-B$. So, by Lemma \ref{lema-indep-1} we have that $|A|\ge i(G-B)\ge i(G)-|B|$ and so, $|B|\ge i(G)-|A|$. Also, for every vertex $u_i\in B$ we have that $|S_i|\ge i(H_i)$ and we have the following,

\begin{align*}
|S|&=\sum_{i=1}^n|S_i|\\
&=\sum_{i=1}^{|A|}|S_i| + \sum_{i=1}^{|B|}|S_i|+\sum_{i=1}^{n-|A|-|B|}|S_i|\\
&\ge \sum_{i=1}^{|A|}i(H)+\sum_{i=1}^{|B|}i(H)+\sum_{i=1}^{n-|A|-|B|}(i(H)-1)\\
&=|A|i(H)+|B|i(H)+(n-|A|-|B|)(i(H)-1)\\
&=n(i(H)-1)+|A|+|B|\\
&\ge n(i(H)-1)+i(G).
\end{align*}
Therefore, the lower bound follows.

To obtain the upper bound, let $A$ be an independent set of maximum cardinality in $G$. Now, for every vertex $u_i\in A$ let $A_i$ be an independent dominating set in $H_i$. Also, for every $u_j\notin A$ let $B_j$ be an independent dominating set in $H_j-v$. Then, it is clear that $\left(\bigcup_{i=1}^{|A|}A_i\right)\cup\left(\bigcup_{j=1}^{n-|A|}B_j\right)\cup A$ is an independent dominating set in $G\circ H$. Therefore the upper bound follows.
\end{proof}

Notice that the above bounds are tight. For instance, if $G$ is the path graph $P_n$ and $H$ is the star graph $S_{1,m}$, $m\ge 2$, with root $v$ in the central vertex, (notice that $G\circ H$ is a caterpillar), then by the above theorem,
\begin{align*}
i(G\circ H)&\le i(S_{1,m})\alpha(P_n) + i(\overline{K_m})(n - \alpha(P_n))\\&=\left\lceil\frac{n}{2}\right\rceil+m\left(n-\left\lceil\frac{n}{2}\right\rceil\right)\\&=mn-\left\lceil\frac{n}{2}\right\rceil(m-1).
\end{align*}

On the contrary, let $S$ be an independent dominating set in $G\circ H$, let $A$ be the set of vertices of $P_n$ belonging to $S$ and let $B_i$, $i\in \{1,...,n\}$, be the set of vertices of $H_i-v$ belonging to $S$. If there is a copy $H_j$ of $H$ in $G\circ H$ such that the root $v$ of $H_j$ belongs to $S$, then neither any vertex of $H_j-v$ nor any neighbor of $v$ in $G$ belongs to $S$. Moreover, if for some copy $H_l$ of $H$ in $G\circ H$ is satisfied that the root $v$ of $H_l$ does not belong to $S$, then every vertex of $H_l-v$ belongs to $S$. Thus,
\begin{align*}
|S|&=|A|+ \sum_{i=1}^{n-|A|}|B_{j_i}|\\&=|A|+m(n-|A|)\\&=mn-|A|(m-1)\\&\ge mn-\alpha(G)(m-1)\\&=mn-\left\lceil\frac{n}{2}\right\rceil(m-1).
\end{align*}

So, $i(G\circ H)=mn-\left\lceil\frac{n}{2}\right\rceil(m-1)$ and the upper bound is tight. To see the sharpness of the lower bound, consider $G$ as a path graph $P_n$ and the graph $H$ obtained from the star graph $S_{1,m}$, $m\ge 2$, by subdividing an edge. Let $v$ be the vertex of $H$ having distance two from the central vertex of the star. If $v$ is the root of $H$, then Theorem \ref{bounds-ind-dom} (ii) leads to
$$i(G\circ H)\ge n(i(H)-1)+i(G)=n(2-1)+\left\lceil\frac{n}{3}\right\rceil=\left\lceil\frac{n}{3}\right\rceil+n.$$

On the other side, let $A$ be the set of all central vertices of all copies of the star $S_{1,m}$, used to obtain $G\circ H$. Since $i(H)=2$ we have that $|A|=n(i(H)-1)$. Let $B$ be an independent dominating set in the path $P_n$. It is clear that $A\cup B$ is an independent dominating set in $G\circ H$. So, $i(G\circ H)\le n(i(H)-1)+i(G)=n+\left\lceil\frac{n}{3}\right\rceil$. As a consequence $i(G\circ H)= n+\left\lceil\frac{n}{3}\right\rceil$ and the lower bound of Theorem \ref{bounds-ind-dom} is achieved.

Moreover, notice that there are graphs in which are not attained any one of the above bounds. The next theorem is an example of that. In order to present such a result we need to introduce some notation. Let $D$ be a set of vertices of a graph $G$, and let $v\in D$. We say that a vertex $x$ is a {\em private neighbor} of $v$ with respect to $D$ if $N[x]\cap D = \{v\}$. The {\em private neighbor set of $v$} with respect to $D$ is $pn[v,D] = N[v] - N[D - \{v\}]$.

\begin{theorem}
Let $G=(V,E)$ be a graph of order $n\ge 2$ and let $H$ be a graph with root $v$ and at least two vertices. Then,
\begin{itemize}
\item[\rm{(i)}] If $v$ does not belong to any $i(H)$-set, then $i(G\circ H)=ni(H)$.
\item[\rm{(ii)}] If $v$ belongs to every $i(H)$-set $S$, then
$$i(G\circ H)\le \alpha(G)i(H) + (n - \alpha(G))(|pn[v,S]| + i(H) - 1).$$
\end{itemize}
\end{theorem}

\begin{proof}
(i) Let us suppose $v$ does not belong to any $i(H)$-set. From Lemma \ref{lemma_G-v_id} we have that $i(H-v)=i(H)$. So, Theorem \ref{bounds-ind-dom} leads to $i(G\circ H)\le ni(H)$. Let $S'$ be a $i(G\circ H)$-set such that $|S'| < ni(H)$. So, there exists at least one copy $H_j$ of $H$ such that $S_j = V_j\cap S'$ and $|S_j| < i(H)$. Since $S_i$ independently dominates $V_i-v$ for every $i\in \{1,...,n\}$, we have that $v$ is not dominated by $S_j$ in $H_j$. Thus, $S_j$ is an independent dominating set in $H_j-v$ and $i(H_j-v)\le |S_j|<i(H)$, which is a contradiction, since $i(H-v) = i(H)$. Therefore, $i(G\circ H)\ge ni(H)$ and the result follows.

(ii) Let $B_i$ be an $i(H_i)$-set, $i\in \{1,...,n\}$ and let $C$ be an independent set of maximum cardinality in $G$. Let $S = \bigcup_{i = 1}^{\alpha (G)} B_i\cup \bigcup_{j=1}^{n - \alpha(G)} (pn[v,B_j] + B_j -\{v\})$. We will show that $S$ is an independent dominating set of $G\circ H$.

Let $B = \bigcup_{i=1}^{n}B_i$. Notice that $B$ is a dominating set in $G\circ H$. If $G = \overline{K_n}$, then $B$ is also independent set in $G\circ H$. In this case $\alpha(G) = n$ and the upper bound follows. Now let us suppose that $G\ncong \overline{K_n}$. Since the root of every copy of $H$ belongs to $B$, there exists at least two roots $v_i$ and $v_j$, $i\ne j$, which are adjacent in $G\circ H$. Thus $B$ is not independent in $G\circ H$.

So, $B' = \bigcup_{i=1}^{\alpha(G)}B_i \cup \bigcup_{\alpha(G) + 1}^{n}(B_i - \{v\})$ is independent set in $H$ and dominates every vertex in $H_i$, except $pn[v, B_i]$. Notice that $B_i-\{v\}$ is still independent in $H_i$, and also, it dominates every vertex in $H_i$, except $pn[v, B_i]$.

Therefore, we have that $i(G\circ H)\le \alpha(G)i(H) + (n - \alpha(G))(|pn[v,S]| + i(H) - 1)$ and the upper bound follows.
\end{proof}

\section{Connected domination number and convex domination number}

A set of vertices $D$ of a graph $G$ is a {\em
connected} \cite{sampathkumar} (or {\em convex} \cite{magda1}) {\em
dominating set} in $G$ if $D$ is a dominating set and the subgraph induced by $D$, (or the set $D$) is connected (or
convex) in $G$. The minimum cardinality of any
connected (or convex) dominating set in $G$ is called
the {\em connected} (or {\em convex}) {\em domination
number} of $G$ and it is denoted by $\gamma_c(G)$ (or $\gamma_{con}(G)$). A set $D$ is a $\gamma_c(G)$-set (or a $\gamma_{con}(G)$-set) if it is a connected (or a convex) dominating set and $|D|=\gamma_c(G)$ (or $|D|=\gamma_{con}(G)$). At next we study the connected (or convex) domination number  of rooted product graphs. We begin with connected domination. This parameter was defined by Sampathkumar and Wallikar in \cite{sampathkumar}.

\begin{theorem}\label{prop1}
Let $G$ be a graph of order $n\ge 2$. Then for any graph $H$ with root $v$ and at least two vertices,
$$\gamma_c(G\circ H)\in \{n\gamma_c(H),n(\gamma_c(H)+1)\}.$$
\end{theorem}

\begin{proof}
Since the vertex $v$ of $H$ is a cut vertex of $G\circ
H$, the vertex $v$ of each copy $H_i$ of $H$ belongs to every
connected dominating set of $G\circ H$. Also, the intersection of
every connected dominating set of $G\circ H$ and the set of vertices of every copy
of $H$ contains a connected dominating set of $H$. So,
$\gamma(G\circ H)\ge \sum_{i=1}^n\gamma_c(H)=n\gamma_c(H)$.

Hence, if $v$ belongs to a $\gamma_c(H_i)$-set $S_i$, then by taking $S=\cup_{i=1}^n S_i$ we have that $S$ is a connected dominating set. So, $\gamma_c(G\circ H)\le \sum_{i=1}^n|S_i|=n\gamma_c(H)$. Therefore, $\gamma_c(G\circ H)=n\gamma_c(H)$.

Now, let us suppose that $\gamma_c(G\circ H)\ne n\gamma_c(H)$. So, $v$ does not belong to any $\gamma_c(H_i)$-set $S_i$. Let $S$ be a $\gamma_c(G\circ H)$-set. If $|S|< n\gamma_c(H)$, then there exists a copy $H_l$ of $H$ in $G\circ H$ in which $|S\cap V_l|<\gamma_c(H)$ and $S\cap V_l$ is a connected dominating set in $H$, which is a contradiction. So, $|S|>n\gamma_c(H)$ and there exists a copy $H_j$ of $H$ such that $|S\cap V_j|>\gamma_c(H)$. Since the root $v$ of $H$ does not belong to any
$\gamma_c(H)$-set, and also $v$ belongs to every $\gamma_c(G\circ H)$-set, we obtain that
$$|S|=\sum_{i=1}^n|S\cap V_i|+|V|\ge n\gamma_c(H)+n=n(\gamma_c(H)+1).$$

On the other hand, let $S_i$ be a $\gamma_c(H_i)$-set, $i\in \{1,...,n\}$. Since $v$ does not belong to any $\gamma_c(G\circ H)$-set, it is satisfied that $v\notin S_i$ for every $i\in \{1,...,n\}$. Thus,  by taking the set $S=V\cup(\bigcup_{i=1}^{n}S_i)$ we have that $S$ is a connected dominating set and, as a consequence,
$$\gamma_c(G\circ H)\le |S|=\sum_{i=1}^n|S_i|+|V|=n\gamma_c(H)+n=n(\gamma_c(H)+1).$$
Therefore, the result follows.
\end{proof}

Next we study the connected domination number of some particular cases of rooted product graphs. We denote by $n_1(G)$ the number of end vertices (vertices of degree one) in $G$ and by $\Omega(G)$ the set of end vertices in $G;$ $|\Omega(G)|=n_1(G).$

\begin{lemma}{\em \cite{sampathkumar}}\label{obs2}
If $T$ is a tree of order at least three, then $\gamma_c(T)=n(T)-n_1(T).$
\end{lemma}

\begin{lemma}\label{rem2}
If $G$ is a connected graph, $H$ is a tree  of order at least three and the root $v$ is non-end vertex of $H,$ then $\gamma_c(G\circ H)=n\gamma_c(H).$
\end{lemma}

\begin{proof} Since the root of the graph $H$ is a cut vertex in the graph $G\circ
H,$ we have that root of each copy $H_i$ of $H$ belongs to every
connected dominating set of $G\circ H$. Also, the intersection of
every connected dominating set of $G\circ H$ and the set of vertices of every copy
of $H$ contains a connected dominating set of $H$. So,
$\gamma_c(G\circ H)\ge \sum_{i=1}^n\gamma_c(H)=n\gamma_c(H)$. Let $D$ be a connected dominating set of $G\circ H.$ Since $H$ is a tree, from Lemma \ref{obs2}, no end vertex belongs to any minimum connected dominating set of $H$ and $\gamma_c(H)=n(H)-n_1(H).$ Also, for every $H_i,$ $\gamma_c(H_i)=n(H_i)-n_1(H_i).$ Since $v$ is non-end vertex of $H,$ we have that $|D|=|V(G)\cup \sum_{i=1}^n(V_i-\Omega_i)|.$ Thus $\gamma_c(G\circ H)\leq |D|=n \gamma(H)$ and we are done.
\end{proof}

\begin{lemma}\label{obs1}
If $T_1, T_2$ are trees  of order at least three, then $T_1 \circ T_2$ is also a tree of order $n(T_1\circ T_2)=n(T_1)n(T_2)$. Moreover, $n_1(T_1\circ T_2)\in \{n(T_1)n_1(T_2), n(T_1)(n_1(T_2)-1)\}.$
\end{lemma}

\begin{proof} For a graph $T_1\circ T_2$ is $n(T_1\circ T_2)=n(T_1)+ n(T_1)(n(T_2)-1)=n(T_1)n(T_2).$ If a root vertex $v$ is an end vertex of $T_2,$ then  $n_1(T_1\circ T_2)= n(T_1)(n_1(T_2)-1).$ On the contrary, if $v$ is not an end vertex of $T_2,$ then $n_1(T_1\circ T_2)= n(T_1)n_1(T_2).$
\end{proof}

\begin{theorem}\label{propt1}
Let $T_1, T_2$ be trees  of order at least three. Then $\gamma_c(T_1\circ T_2)=n(T_1) \gamma_{c}(T_2)$ if and only if the rooted vertex $v$ of $T_2$ is not an end vertex of $T_2$.
\end{theorem}

\begin{proof}
From Remarks \ref{obs2} and \ref{obs1}, we have $\gamma_c(T_1 \circ T_2)=n(T_1\circ T_2)-n_1(T_1\circ T_2).$  Also from Lemma \ref{obs1} we have $n(T_1\circ T_2)=n(T_1)n(T_2).$ Let $v$ be a non end vertex of $T_2$. Hence, $n_1(T_1\circ T_2)=n(T_1)n_1(T_2).$ Thus $\gamma_c(T_1 \circ T_2)=n(T_1)n(T_2)-n_1(T_2)n(T_1)$ $=n(T_1)(n(T_2)-n_1(T_2))=n(T_1)\gamma_c(T_2).$

Assume now  $\gamma_c(T_1\circ T_2)=n(T_1) \gamma_{c}(T_2)$ and suppose $v$ is an end vertex of $T_2.$ Hence we have $n_1(T_1 \circ T_2)=(n_1(T_2)-1)n(T_1)=n_1(T_2)n(T_1)-n(T_1).$ Since $n(T_1\circ T_2)=n(T_1)n(T_2),$  we have $\gamma_c(T_1\circ T_2)=n(T_1)n(T_2)-(n_1(T_2)n(T_1)-n(T_1))$ $=n(T_1)(n(T_2)-n_1(T_2)+1)=n(T_1)(\gamma_c(T_2)+1),$ what gives a contradiction.
\end{proof}

From Theorems \ref{prop1} and \ref{propt1} we can conclude the following.

\begin{corollary}
$\gamma_c(T_1\circ T_2)=n(T_1)(\gamma_c(T_2)+1)$ if and only if the rooted vertex $v$ of $T_2$ is an end vertex of $T_2$.
\end{corollary}

Convex domination was defined by Topp in \cite{topp} and it was first characterized in \cite{magda1}.
Notice that for the case of the convex domination number of $G\circ
H$ the result is similar to Theorem \ref{prop1} about connected
domination. The proofs of the following results are omitted due to the
analogy with the above ones.

\begin{theorem}
Let $G$ be a graph of order $n\ge 2$. Then for any graph $H$ with root $v$ and at least two vertices,
$$\gamma_{con}(G\circ H)\in \{n\gamma_{con}(H),n(\gamma_{con}(H)+1)\}.$$
\end{theorem}

\begin{theorem}\label{propt2}
If $T_1, T_2$ are trees, then $\gamma_{con}(T_1\circ T_2)=n(T_1) \gamma_{con}(T_2)$ if and only if the rooted vertex $v$ is not an end vertex of a tree $T_2.$
\end{theorem}

\begin{corollary}
$\gamma_{con}(T_1\circ T_2)=n(T_1)(\gamma_{con}(T_2)+1)$ if and only if the rooted vertex $v$ is an end vertex of $T_2.$
\end{corollary}

\subsection{Weakly connected domination number}

Now we consider the weakly connected domination number of rooted product graphs. A dominating set $D\subset V(G)$ is a {\it weakly connected dominating set} in $G$ if the subgraph $G[D]_{w}=(N_{G}[D],E_{w})$ (also called subgraph weakly induced by $D$) is connected, where $E_{w}$ is the set of all edges having at least one vertex
in $D$. Dunbar et al. \cite{Dunbar} defined the {\it weakly connected domination number} $\gamma_{w}(G)$ of a graph $G$ to be the minimum
cardinality among all weakly connected dominating sets in $G$.

\begin{theorem}
Let $G$ be a graph of order $n\ge 2$. Then for any graph $H$ with root $v$ and at least two vertices,
$$\gamma_w(G\circ H)\in \{n\gamma_w(H), n\gamma_w(H)+\gamma_w(G)\}.$$
\end{theorem}

\begin{proof} Let $D_H$ be a minimum weakly connected dominating set of $H$ and $D_{H_i}$ be the copy of $D_H$ in the $i^{th}$ copy $H_i$ of $H, 1\leq i\leq n.$ Let $D$ be a minimum weakly connected dominating set of $G\circ H.$ We consider two cases.
\begin{enumerate}
\item $v \in D_H.$ Then identified vertices belong to a minimum weakly connected dominating  set of $G\circ H$ and $\gamma_w(G\circ H)=n \gamma_w(H).$
\item $v\notin D_H.$ Then $\bigcup_{i=1}^n D_{H_i}\subset D$ and identified vertices are dominated by $\bigcup_{i=1}^n D_{H_i}.$ But the set $\bigcup_{i=1}^n D_{H_i}$ is not weakly connected. To make this set weakly connected, we need to add to this set $\gamma_w(G)$ vertices. So $\gamma_w(G\circ H)=|D|=|\bigcup_{i=1}^n D_{H_i}|+\gamma_w(G)=n\gamma_w(H)+\gamma_w(G)$.
\end{enumerate}
\end{proof}

The following lemma presented in \cite{magdal1} will be useful into obtaining some interesting results.

\begin{lemma}\label{magda1}
For any tree $T$ of order $n\ge 3$,
$$\frac{1}{2}(n-n_1(T)+1)\leq \gamma_w(T)\leq n-n_1(T).$$
\end{lemma}

\begin{theorem}
If $T_1,T_2$ are trees and $v$ is not an end vertex of $T_2$, then $$\frac{1}{2}(n_1(T_1)\gamma_w(T_2)+1)\leq \gamma_w(T_1\circ T_2)\leq n_1(T_1)(2 \gamma_w(T_2)-1).$$
\end{theorem}

\begin{proof} From Lemma \ref{magda1},  $\frac{1}{2}(n(T_1\circ T_2)-n_1(T_1\circ T_2)+1)\leq \gamma_w(T_1\circ T_2)\leq n(T_1\circ T_2)-n_1(T_1\circ T_2).$ Thus, from Lemma \ref{obs1}, we have  $\frac{1}{2}(n(T_1)n(T_2)-n(T_1)n_1(T_2)+1)\leq \gamma_w(T_1\circ T_2)\leq n_1(T_1)(n(T_2)-n_1(T_2)).$
We have $\gamma_w(T_1\circ T_2)\leq n_1(T_1)(n(T_2)-n_1(T_2))=n_1(T_1)2 \frac{1}{2}(n(T_2)-n_1(T_2))=$ $n_1(T_1) 2 \frac{1}{2}(n(T_2)-n_1(T_2)+1-1)\leq n_1(T_1) 2 \gamma_w(T_2)-n_1(T_1)=n_1(T_1)(2 \gamma_w(T_2)-1).$
From the other side we have $\gamma_w(T_1\circ T_2)\geq \frac{1}{2}(n_1(T_1)(n(T_2)-n_1(T_2))+1)\geq \frac{1}{2}(n_1(T_1)\gamma_w(T_2)+1)$ and finally we obtain the desired result.
\end{proof}

By using similar methods, we obtain the following result.

\begin{theorem}
Let $T_1$ be a tree of order $n(T_1)$. If $v$ is a non-end vertex of a tree $T_2$, then $$\frac{1}{2}(\gamma_w(T_2)n(T_1)+1)\leq \gamma_w(T_1\circ T_2)\leq 2 n(T_1)\gamma_w(T_2).$$
\end{theorem}

\subsection{Super domination number}

We continue with the super domination number of the rooted product graph. This parameter was defined in \cite{magda}. A subset $D$ of $V$ is called a {\it super dominating set} if for every $v\in V-D$ there exists $u\in N_G(v)\cap D$ such that $N_G(u)\subseteq D\cup \{v\}.$ The minimum cardinality of a super dominating set is called the {\it super domination number} of $G$ and it is denoted by $\gamma_{sp}(G)$. In \cite{magda} a paper was proved the following result.

\begin{lemma}{\em\cite{magda}}\label{magda2}
For any tree of order $n\ge 3$, $\frac{n}{2}\leq \gamma_{sp}(T_1\circ T_2)\leq n-s(T),$ where $s(T)$ is the number of support vertices in $T.$
\end{lemma}

\begin{theorem}
Let $G$ be a graph of order $n\ge 2$. Then for any graph $H$ with root $v$ and at least two vertices,
$$\gamma_{sp}(G\circ H)= n\gamma_{sp}(H).$$
\end{theorem}

\begin{proof} Let $D_H$ be a minimum super dominating set of $H$ and $D_{H_i}$ be the copy of $D_H$ in the $i^{th}$ copy $H_i$ of $H, 1\leq i\leq n.$ Let $D$ be a minimum super dominating set of $G\circ H.$  We consider two cases.
\begin{enumerate}
\item $v\in D_H.$ Then identified vertices belong to a minimum super dominating set of $G\circ H$ and $\gamma_{sp}(G\circ H)=n \gamma_{sp}(H).$
\item $v\notin D_H.$ Then $\bigcup_{i=1}^n D_{H_i}\subset D$ and identified vertices are dominated by $\bigcup_{i=1}^n D_{H_i}.$ Also,  every vertex belonging to $\bigcup_{i=1}^n V_i - \bigcup_{i=1}^nD_{H_i} - U,$ where $U$ is the set of identified vertices is super dominated by $\bigcup_{i=1}^n D_{H_i}.$ Suppose there exists a vertex $u\in U$ which is not super dominated by $\bigcup_{i=1}^n D_{H_i}.$ Then for a vertex $u$ does not exist any $v\in \bigcup_{i=1}^n D_{H_i}$ such that $N_{G\circ H}(v)\subseteq \bigcup_{i=1}^n D_{H_i}.$ Then there exists $1\leq i \leq n$ such that in $H_i$ there exists a vertex which is not super dominated, a contradiction. Thus $\bigcup_{i=1}^n D_{H_i}$ is a minimum super dominating set of $G\circ H$ and $\gamma_{sp}(G\circ H)=|\bigcup_{i=1}^n D_{H_i}|=n\gamma_{sp}(H).$
\end{enumerate}
\end{proof}

Now, by using Lemma \ref{magda2} result, we can prove the following.

\begin{theorem}
If $T_1, T_2$ are trees of order $n(T_1)\ge 3$ and $n(T_2)\ge 3$, respectively, then
$$n(T_1)s(T_2)\leq \gamma_{sp}(T_1\circ T_2)\leq n(T_1)(n(T_2)-s(T_2)).$$
\end{theorem}

\begin{proof}
If a root $v$ is a support vertex of $T_2,$ then $s(T_1\circ T_2)=n(T_1)+(s(T_2)-1)n(T_1)=n(T_1)s(T_2).$ If $v$ is not a support vertex, then also $s(T_1\circ T_2)=n(T_1)s(T_2).$ The upper bound holds directly from  Lemma \ref{magda2}. For the lower bound we have $\gamma_{sp}(T_1\circ T_2)\leq 2 \gamma_{sp}(T_1\circ T_2)-n(T_1)s(T_2).$ From this inequality we obtain the final lower bound.
\end{proof}


\begin{thebibliography}{99}

\bibitem{berge} C. Berge, \emph{Graphs and hypergraphs}. North-Holland, Amsterdam. 1973

\bibitem{upper-dom-cart} B. Bre\v{s}ar, Vizing-like conjecture for the upper domination of Cartesian products of graphs - the proof, \emph{Electrononic Journal of Combinatorics} \textbf{12} (2005) \#N12.

\bibitem{dom-direct} B. Bre\v{s}ar, S. Klav\v{z}ar, D. F. Rall, Dominating direct products of graphs, \emph{Discrete Mathematics} \textbf{307} (2007) 1636--1642.

\bibitem{roman} E. J. Cockayne, P. A. Dreyer, S M. Hedetniemi, S. T. Hedetniemi, Roman domination in graphs, {\em Discrete Mathematics} {\bf
278} (1-3) (2004) 11--22.

\bibitem{tot-dom-direct} P. Dorbec, S. Gravier, S. Klav\v{z}ar, S. \v{S}pacapan, Some results on total domination in direct product graphs, \emph{Discussiones Mathematicae Graph Theory} \textbf{26} (2006 ) 103--112.

\bibitem{Dunbar} J. Dunbar, J. Grossman, S. Hedetniemi, J. Hatting, A. McRae, On weakly-connected domination in graphs, {\em Discrete Mathematics} {\bf 167-168} (1997) 261--269.

\bibitem{rooted-first} C. D. Godsil, B. D. McKay, A new graph product and its spectrum, {\em Bulletin of the Australasian Mathematical Society}
{\bf 18} (1) (1978) 21--28.

\bibitem{dom-par-corona} I. Gonz\'alez Yero, D. Kuziak, A. Rond\'on Aguilar, Coloring, location and domination of corona graphs, manuscript. {\em http://arxiv.org/pdf/1204.0647.pdf}

\bibitem{bookdom1} T. W. Haynes, S. T. Hedetniemi, P. J. Slater, \emph{Fundamentals of Domination
in Graphs}, Marcel Dekker, Inc. New York, 1998.

\bibitem{ind-dom-kron} P. K. Jha, Smallest independent dominating sets in Kronecker products of cycles, \emph{Discrete Applied Mathematics} \textbf{113} (2001) 303--306.

\bibitem{dom-direct-1} S.  Klav\v{z}ar,  B. Zmazek, On  a  Vizing-like  conjecture  for  direct  product  graphs, \emph{Discrete  Mathematics}  \textbf{156}  (1996)  243--246.

\bibitem{Rom-dom-lexicog} T. Kraner \v{S}umenjaka, P. Pavli\v{c}, A. Tepeh, On the Roman domination in the lexicographic product of graphs, \emph{Discrete Applied Mathematics} \textbf{160} (13-14) (2012) 2030--2036.

\bibitem{magdal1} M. Lema\'nska, Lower bound on the weakly connected domination number of a tree,
{\em Australasian Journal of Combinatorics}, {\bf 37} (2007) 67--71.

\bibitem{magda} M. Lema\'nska, V. Swaminathan, Y. B. Venkatakrishnan, R. Zuazua, Super dominating sets in graphs. Manuscript.

\bibitem{magda1} M. Lema\'nska, Weakly convex and convex domination numbers,
\emph{Opuscula Mathematica} \textbf{24} (2) (2004) 181--188.

\bibitem{dom-direct-2} D. F. Rall, Total domination in categorical products of graphs, \emph{Discussiones Mathematicae Graph Theory} \textbf{25} (2005) 35--44.

\bibitem{sampathkumar} E. Sampathkumar, H. Walikar, The connected domination number of a graph, {\em Journal of Mathematical and Physical Sciences} {\bf 13} (6) (1979) 607--613.

\bibitem{roman-1} I. Stewart, Defend the Roman Empire!, {\em Scientific American}, December (1999) 136--138.

\bibitem{topp} J. Topp, Private communication. (2002).

\bibitem{vizing1} V. G. Vizing, The Cartesian product of graphs, {\it Vy\v{c}isl. Sistemy} {\bf 9} (1963)
 30--43.

\bibitem{vizing} V. G. Vizing, Some unsolved problems in graph theory, {\it Uspehi Mathematika Nauk} {\bf 23} (144) (1968)
 117--134.

\bibitem{dom-par-conjunction} M. Zwierzchowski, A note on domination parameters of the conjunction of two special graphs, \emph{Discussiones Mathematicae Graph Theory} \textbf{21} (2001) 303--310.

\end{thebibliography}
\end{document}